\documentclass[12pt]{amsart}
\usepackage{a4}
\usepackage{amssymb}
\usepackage{amsmath}
\usepackage{amsthm}
\usepackage{amstext}
\usepackage{amscd}
\usepackage{latexsym}
\usepackage{graphics}
\usepackage{color}

\swapnumbers
\theoremstyle{plain}
\newtheorem{thm}{Theorem}

\newtheorem{lem}[thm]{Lemma}

\theoremstyle{definition}
\newtheorem{rem}[thm]{Remark}

\newtheorem{ex}[thm]{Example}

\newcommand{\C}{{\mathbb C}}

\renewcommand{\P}{{\mathbb P}}

\newcommand{\Z}{{\mathbb Z}}

\input{xy}
\xyoption{all}

\begin{document}
\title{Fano Varieties with Large Degree  Endomorphisms}
\author{J\'anos Koll\'ar and Chenyang Xu}

\maketitle


\vspace*{6pt}

We work over the complex number field $\C$. An {\it endomorphism} of a projective variety $X$ is a {\it morphism} $f:X\to X$. If $f$ is surjective, then it is automatically finite.

Examples of projective varieties with an endomorphism of degree $>1$
are $\P^n$ and abelian varieties. In fact, these are essentially all known
smooth examples.
When $X$ has non-negative Kodaira dimension, this is discussed in \cite{nz07}.
 It is also conjectured that
if $X$ is a smooth Fano variety of Picard number 1 that
 admits a degree $>1$ endomorphism, then $X\cong\P^n$  \cite{am97}.
This  is confirmed in some cases, including when the dimension is less or equal to 3
 \cite{arv99}, when $X$ is a quasi-homogeneous \cite{hn08}, when $X$ is a large degree hypersurface of a prime Fano manifold of Picard number 1 \cite{ch08}, or when $X$ contains a smooth conic \cite{hm03}.
\cite{zh08} considered Fano varieties with Picard number 1 and,
with terminal singularities that have an endomorphism of degree
$>1$. He conjectured that they are necessarily rational. Our aim
in this note is to point out that this is not quite true and
suggest a slight modification that is consistent with these
examples.

\begin{ex}\label{construction} Any subgroup  $G\subset S_{n+1}$ acts
 on $\P^{n}$ where $S_{n+1}$ acts  by permuting the coordinates.
 The endomorphism $\phi_d$
of raising each coordinate to its $d$-th power commutes with
 this action, so it descends to a degree $d^{n}$  endomorphism of
$X:=\P^{n}/G$.

According to the philosophy of \cite{kl07}, $\P^{n}/G$ should have
terminal singularities for almost all $G$. The precise conditions are
checked in Lemma \ref{terminal}. Some examples where
$\P^{n}/G$ is not rational are given by  \cite{sa85}.
In his examples, $G$ is a nonabelian group  of order $p^9$ for some prime $p$
and $G\subset S_{|G|+1}$ fixes  the last coordinate
and acts by the left $G$-action on itself on the
first $|G|$  coordinates.
\end{ex}

\begin{rem} 1. These examples suggest  the possibility  that if  a
Fano variety $X$ of Picard number 1 has terminal singularities and
admits a degree $>1$ endomorphism then $X$ is  a
quotient of $\P^n$.

2. It seems likely that if $G$ is large enough, then the $\phi_d$
are the only endomorphisms of $\P^n$ that commute with the $G$-action.
Thus we may get examples of Fano varieties which have
very few endomorphisms of degree $>1$.
\end{rem}

\begin{lem}\label{terminal}
 Let $G\subset S_{n}$ be  a subgroup with the induced action on $\P^{n-1}$
 by permuting the coordinates. Assume that none of the elements
 of $G$ is conjuagte (in $ S_{n}$) to one of the premutations
$(12), (123)$ or $(12)(34)$.
Then $\P^{n-1}/G$ has terminal singularities.
\end{lem}

\begin{proof}
 Let the order of $g$ be $m$. Assume that the action of $g\in G$ on the set
 $\{1,2,...,n\}$ gives cycles of length $n_1,n_2,...,n_k$ with
$\sum_{i=1}^kn_i=n$. So $n_i|m$ and we write $m=n_im_i$.
After a linear transformation of the coordinates, $g$ acts  by
$$(t_1,t_2,...,t_n) \to
(t_1,\xi^{m_1}t_2,...,\xi^{(n_1-1)m_1}t_{n_1},...,\xi^{(n_k-1)m_k}t_n),$$
where $\xi$ is a primitive $m$-th root of unity.

Now we compute the {\it age} of $g$ at a fixed point $x$ of $g$.
(See \cite{ir96}, \cite{re02} for the definition and for
the  Reid-Tai criterion.)
If a fixed point is on the chart $t_p\neq 0$, then, for each $1\le j\le k$,
the contribution of the cycle of length $n_j$  to the age is
 the sum of all numbers $c/m$ where $0\le c <m$ and  $c$  can be written as
 $c=km_j-p$ for some $k\in \Z$. This  sum is bigger or equal to
$$
\frac{m_j}{m}(0+\cdots+(n_j-1)\bigr)=\frac{n_j-1}{2}.
$$
Therefore,
${\rm age}_x(g)\ge\sum_{i=1}^k(n_i-1)/2$
 is always larger than 1 if $g$ is not any one of the cases  listed above.
 The Reid-Tai criterion implies that $\P^{n-1}/G$ has terminal singularities.
\end{proof}

\end{document}